\documentclass{amsart}
\usepackage{amsmath, amsthm, amssymb, amscd, mathrsfs, eucal, epsfig}
\usepackage{pinlabel}

\begin{document}

\newtheorem{theorem}{Theorem}[section]
\newtheorem{lemma}[theorem]{Lemma}
\newtheorem{corollary}[theorem]{Corollary}
\newtheorem{conjecture}[theorem]{Conjecture}
\newtheorem{question}[theorem]{Question}
\newtheorem{problem}[theorem]{Problem}
\newtheorem*{claim}{Claim}
\newtheorem*{criterion}{Criterion}
\newtheorem*{main_thm}{Theorem A}

\theoremstyle{definition}
\newtheorem{definition}[theorem]{Definition}
\newtheorem{construction}[theorem]{Construction}
\newtheorem{notation}[theorem]{Notation}

\theoremstyle{remark}
\newtheorem{remark}[theorem]{Remark}
\newtheorem{example}[theorem]{Example}

\def\Z{\mathbb Z}
\def\homeo{\textnormal{Homeo}}
\def\diff{\textnormal{Diff}}
\def\inte{\textnormal{int}}
\def\fix{\textnormal{fix}}
\def\CP{\mathbb{CP}}
\def\R{\mathbb R}
\def\RP{\mathbb{RP}}
\def\r{{\bf r}}
\def\Id{\textnormal{Id}}
\def\SL{\textnormal{SL}}
\def\PSL{\textnormal{PSL}}
\def\length{\textnormal{length}}
\def\til{\tilde}

\title{Nonsmoothable, locally indicable group actions on the interval}
\author{Danny Calegari}
\address{Department of Mathematics \\ Caltech \\
Pasadena CA, 91125}
\email{dannyc@its.caltech.edu}

\date{2/19/2008, Version 0.04}

\begin{abstract}
By the Thurston stability theorem, a group of $C^1$ orientation-preserving diffeomorphisms
of the closed unit interval is locally indicable. We show that the local order structure of orbits
gives a stronger criterion for nonsmoothability that can be used to produce new examples of 
locally indicable groups of homeomorphisms
of the interval that are not conjugate to groups of $C^1$ diffeomorphisms.
\end{abstract}

\maketitle

\section{Introduction}
\subsection{Acknowledgment}
This note was inspired by a comment in a lecture by Andr\'es Navas. I would like to thank
Andr\'es for his encouragement to write it up. I would also like to thank the referee,
whose many excellent comments have been incorporated into this paper.

\section{Nonsmoothable actions}

\subsection{Thurston stability theorem}
A simple, but important case of the Thurston Stability Theorem is usually stated in
the following way:

\begin{theorem}[Thurston Stability Theorem \cite{Thurston}]\label{tst}
Let $G$ be a group of orientation-preserving $C^1$ diffeomorphisms of the closed interval
$I$. Then $G$ is locally indicable; i.e. every nontrivial finitely generated subgroup $H$ of $G$
admits a surjective homomorphism to $\Z$.
\end{theorem}

The proof is non-constructive, and uses the axiom of choice. The idea is to ``blow up'' the
action of $H$ near one of the endpoints at a sequence of points that are moved a definite
distance, but not too far. Some subsequence of blow-ups converges to an action by translations.

Note that it is only {\em finitely} generated subgroups that 
admit surjective homomorphisms to $\Z$, as the following example of Sergeraert shows.

\begin{example}[Sergeraert \cite{Sergeraert}]
Let $G$ be the group of $C^\infty$ orientation-preserving diffeomorphisms of $I$ that
are infinitely tangent to the identity at the endpoints. Then $G$ is perfect.
\end{example}

Another countable example comes from Thompson's group.

\begin{example}[Navas \cite{Navas2}, Ghys-Sergiescu \cite{Ghys_Sergiescu}]
Thompson's group $F$ of dyadic rational piecewise linear homeomorphisms of $I$ is known
to be conjugate to a group of $C^\infty$ diffeomorphisms. On the other hand, the
commutator subgroup $[F,F]$ is simple; since it is non-Abelian, it is perfect.
\end{example}

Given a group $G \subset \homeo_+(I)$, Theorem~\ref{tst} gives a criterion to show that the action of
$G$ is not conjugate into $\diff^1_+(I)$. It is natural to ask
whether Thurston's criterion is sharp. That is, suppose $G$ is
locally indicable. Is it true that every homomorphism from $G$ into $\homeo_+(I)$ is conjugate
into $\diff^1_+(I)$? It turns out that the answer to this question is no. However, apart from
Thurston's criterion, very few
obstructions to conjugating a subgroup of $\homeo_+(I)$ into $\diff^1_+(I)$ are known.
Most significant are dynamical obstructions concerning the existence of elements with
hyperbolic fixed points when the action has positive topological entropy \cite{Hurder},
or when there is no invariant probability measure \cite{DKN} (also, see \cite{Cantwell_Conlon}).

In this note we give some new examples of actions of locally indicable groups
on $I$ that are not conjugate to $C^1$ actions. 

\begin{example}[$\Z^\Z$]\label{abelian}
Let $T:I \to I$ act freely on the interior, so that $T$ is conjugate to a translation.
Let $I_0 \subset \inte(I)$ be a closed fundamental domain for $T$, and let $S:I_0 \to I_0$
act freely on the interior. Extend $S$ by the identity outside $I_0$ to an element of $\homeo_+(I)$.
For each $i \in \Z$ let $I_i = T^i(I_0)$ and let
$S_i:I_i \to I_i$ be the conjugate $T^i S T^{-i}$.
For each $f \in \Z^\Z$ define $Z_f$ to be the product
$$Z_f = \prod_{i \in \Z} S_i^{f(i)}$$
Let $G$ be the group consisting of all elements of the form $Z_f$. Then $G$ is isomorphic to
$\Z^\Z$ and is therefore abelian. 

However, $G$ is not conjugate into $\diff_+^1(I)$. For, suppose otherwise, so that there is
some homeomorphism $\varphi:I \to I$ so that the conjugate $G^\varphi \subset \diff_+^1(I)$.
We suppose by abuse of notation that $S_i$ denotes the conjugate $S_i^\varphi$.
For each $i$, let $p_i$ be the midpoint of $I_i$.
Since for each fixed $i$ the sequence $S_i^n(p_i)$ converges to an endpoint
of $I_i$ as $n$ goes to infinity, it follows that for each $i$ there is some $n_i$ so
that $dS_i^{n_i}(p_i) < 1/2$. Let $F \in \Z^\Z$ satisfy $F(i) = n_i$. Then
$dZ_F(p_i) < 1/2$ for all $i$. However, $Z_F$ fixes the endpoints of $I_i$ for all $i$, so $Z_F$
has a sequence of fixed points converging to $1$. It follows that $dZ_F(1) = 1$. But
$p_i \to 1$, so if $Z_F$ is $C^1$ we must have $dZ_F(1) \le 1/2$. This contradiction shows that
no such conjugacy exists.
\end{example}

\begin{remark}
The group $\Z^\Z$ is locally indicable, but uncountable. Note in fact that this group
action is not even conjugate to a {\em bi-Lipschitz} action. On the other hand,
Theorem~D from \cite{DKN} says that every countable group of homeomorphisms of the
circle or interval is conjugate to a group of bi-Lipschitz homeomorphisms.
\end{remark}

\subsection{Order structure of orbits}

In this section we describe a new criterion for non-smoothability, depending on the local
order structure of orbits.

\begin{definition}
Let $G$ act on $I$ by $\rho:G \to \homeo_+(I)$. A point $p \in I$ determines an order $<_p$
on $G$ by 
$$a <_p b$$
if and only if $a(p) < b(p)$ in $I$.
\end{definition}

Note that with this definition, $<_p$ is really an order on the left $G$-space 
$G/G_p$, where $G_p$ denotes the stabilizer of $p$.

\begin{lemma}\label{order_structure}
Suppose $\rho:G \to \diff_+^1(I)$ is injective. Let $H$ be a finitely generated subgroup
of $G$, with generators $S = \lbrace h_1,\cdots,h_n \rbrace$. 
Let $p \in I$ be in the frontier of $\fix(H)$
(i.e. the set of common fixed points of all elements of $H$) and let $p_i \to p$ be
a sequence contained in $I - \fix(H)$. Then 
there is a sequence $k_m \in \lbrace 1, \dots, n\rbrace$ and $e_m \in \pm 1$
such that for any $h \in [H,H]$, and for all sufficiently large $m$ (depending on $h$), there
is an inequality
$$h <_{p_m} h_{k_m}^{e_m}$$
\end{lemma}
\begin{proof}
There is a homomorphism $\rho:H \to \R$ defined by the formula $\rho(h) = \log h'(p)$.
Of course this homomorphism vanishes on $[H,H]$. If $h_i$ is such that
$\rho(h_i) \ne 0$ then (after replacing $h_i$ by $h_i^{-1}$ if necessary) it is
clear that for any $h \in [H,H]$, there is an inequality $h <_{p_m} h_i$ for all 
$p_m$ sufficiently close to $p$. Therefore in the sequel we assume $\rho$ is
trivial.

For each $i$, let
$U_i$ be the smallest (closed) interval containing $p_i \cup Sp_i$.
Given a bigger open interval $V_i$ containing $U_i$, one can rescale $V_i$ linearly
by $1/\length(U_i)$ and move $p_i$ to the origin thereby obtaining an interval 
$\overline{V}_i$ on which $H$ has a partially defined action as a pseudogroup.

The argument of the Thurston stability theorem implies that one can choose a sequence
$V_i$ such that any sequence of indices $ \to \infty$ contains a subsequence for which
$\overline{V}_i \to \R$, and the pseudogroup actions converge, in the compact-open topology,
to a (nontrivial) action of $H$ on $\R$ by {\em translations}. In an action by
translations, some generator or its inverse moves $0$ a positive distance, 
but every element of $[H,H]$ acts trivially. The proof follows.
\end{proof}

\begin{example}
Let $T$ be a hyperbolic once-punctured torus with a cusp. The hyperbolic
structure determines up to conjugacy a faithful homomorphism $\rho:\pi_1(T) \to \PSL(2,\R)$.

The group $\PSL(2,\R)$ acts by real analytic homeomorphisms on $\RP^1 = S^1$.
Since $\pi_1(T)$ is free on two generators (say $a,b$) the homomorphism $\rho$ lifts
to an action $\tilde{\rho}$ on the universal cover $\R$. We choose a lift so that both $a$ and
$b$ have fixed points. If we choose co-ordinates on $\R$ so that $a$ fixes $x$,
then $a$ also fixes $x+n$ for every integer $n$. Similarly, if $b$ fixes $y$, then
$b$ fixes $y+n$ for every $n$. On the other hand, if $p \in S^1$ is the parabolic
fixed point of $[a,b]$, and $\til{p}$ is a lift of $p$ to $\R$,
then the commutator $[a,b]$ takes $\til{p}$ to $\til{p}+1$. Since the action of every element
on $\R$ commutes with the generator of the deck group $x \to x+1$, the element $[a,b]$ acts on $\R$
without fixed points, and moves every point in the positive direction, satisfying
$[a,b]^n(z) > z + n -1$ for every $z \in \R$ and every positive integer $n$. See Figure~\ref{circle_action}.

\begin{figure}[ht]
\labellist
\small\hair 2pt
\pinlabel $p$ at 30 300
\pinlabel $s$ at 30 80
\pinlabel $q$ at 260 300
\pinlabel $r$ at 260 80
\pinlabel $a$ at 110 200
\pinlabel $b$ at 155 230
\pinlabel $\til{p}$ at 480 40
\pinlabel $\til{q}$ at 480 110
\pinlabel $\til{r}$ at 480 183
\pinlabel $\til{s}$ at 480 255
\pinlabel $\til{p}+1$ at 460 325
\pinlabel $a$ at 560 75
\pinlabel $b$ at 560 146
\pinlabel $a^{-1}$ at 570 220
\pinlabel $b^{-1}$ at 570 292
\endlabellist
\centering
\includegraphics[scale=0.4]{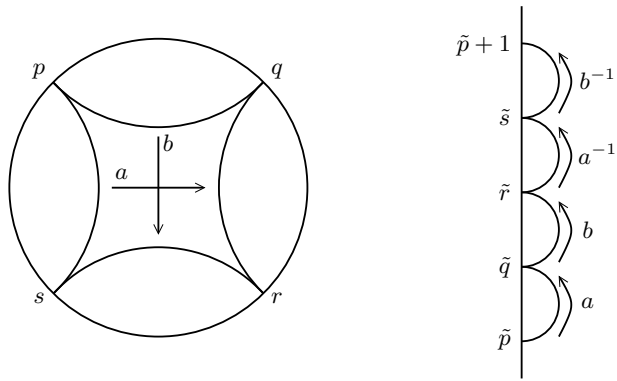}
\caption{In the lifted action, $a$ and $b$ have fixed points, but $[a,b]$
takes $\til{p}$ to $\til{p}+1$}\label{circle_action}
\end{figure}

This action on $\R$ can be made into an action on $I$ by homeomorphisms, by including
$\R$ in $I$ as the interior. Then the points $\til{p}+n \to \infty$ in $\R$ map to points
$p_n \to 1$ in $I$. Note that for each $n$, the elements $a$ and $b$ have fixed points $q_n,r_n$
respectively satisfying $p_n < q_n < p_{n+1}$ and $p_n < r_n < p_{n+1}$. 
Moreover, $[a,b](p_n) = p_{n+1}$ for all $n$. It follows that
$$a, a^{-1} <_{p_n} [a,b]^2, \quad b, b^{-1} <_{p_n} [a,b]^2$$
for every $n$, so by Lemma~\ref{order_structure},
this action is not topologically conjugate into $\diff_+^1(I)$. On the other hand,
this is a faithful action of the free group on two generators. A free group is locally
indicable, since every subgroup of a free group is free.
\end{example}

\begin{remark}
The relationship between order structures and dynamics of subgroups of
homeomorphisms of the interval is subtle and deep. For an introduction to this subject,
see e.g. \cite{Navas}.
\end{remark}

\end{document}